\documentclass{amsart}
\usepackage{amssymb}
\newtheorem{thm}{Theorem}[section]

\newtheorem{lem}[thm]{Lemma}

\theoremstyle{definition}

\theoremstyle{remark}

\numberwithin{equation}{section}
\newcommand{\R}{\mathbb R}
\newcommand{\eps}{\varepsilon}
\newcommand{\li}{\langle}
\newcommand{\ri}{\rangle}

\newcommand{\rt}{\rightarrow}

\begin{document}

\title[rigidity theorems]{An elementary approach to\\ some rigidity theorems}
\author{harish seshadri}
\address{department of mathematics,
Indian Institute of Science, Bangalore 560012, India}
\email{harish@math.iisc.ernet.in}

%\date{}%
%\dedicatory{}%
%\commby{}%
% ----------------------------------------------------------------
\begin{abstract}

%\vspace{2mm}
Using elementary comparison geometry, we prove:

Let $(M,g)$ be a simply-connected complete Riemannian manifold of
dimension $\ge 3$. Suppose that the sectional curvature $K$
satisfies
 \ $ -1-s(r) \le K \le -1$, \ where $r$ denotes distance to a fixed point in
 $M$. If \ $\lim_{ \ r \rt \infty} \ e^{2r}s(r) =0$, \ then $(M,g)$ has
to be isometric to ${\mathbb H}^n$.

The same proof also yields that if $K$ satisfies \ $-s(r) \le K
\le 0$ \ where \ $\lim_{r \rt \infty} r^2s(r)=0$, \ then $(M,g)$
is isometric to $\R^n$, \ a result due to Greene and Wu.

%\vspace{2mm}

Our second result is a local one: Let $(M,g)$ be any Riemannian
manifold. For $a \in \R$, \ if \ $K \le a$ \ on a geodesic ball
$B_p(R)$ in $M$ and \ $K = a$ on \ $\partial B_p(R)$, \ then $K=
a $ on \ $B_p(R)$.

\end{abstract}

\thanks{Mathematics Subject Classification (1991): Primary 53C21, Secondary 53C20}

\maketitle
% ----------------------------------------------------------------
\section{Introduction}

The question of when a Riemannian manifold which asymptotically
``resembles" $\R^n$ is actually isometric to $\R^n$ is a
classical topic in differential geometry. Broadly speaking,
attention has been focused on two notions of resemblance. In the
first, one makes a weak curvature assumption, such as the
nonpositivity of scalar curvature, but one also assumes the
existence of a coordinate system (outside a compact set) in which
the metric approximates the standard Euclidean metric. The
Positive Mass Theorem of Schoen and Yau ~\cite{scy} is the
prototype of such a result.  In the second class, it is assumed
that the sectional curvature has a definite sign and approaches
zero at a certain rate. One of the early results in this
direction was by Siu and Yau ~\cite{sy}. One of the byproducts of
this paper is a completely elementary and short proof of the main
result in ~\cite{sy}. A host of theorems was also proved by
Greene and Wu in ~\cite{gw}. In particular, they proved:
\vspace{3mm}

\noindent {\bf Theorem 1 (Greene-Wu ~\cite{gw})}: {\it Let $(M,g)$
be a simply-connected complete Riemannian manifold of dimension
$\ge 3$. Suppose that \ $ -s(r) \ \le \ K \ \le \ 0$, where $r$
denotes distance to a fixed point in $M$.

%$ -s(r)
%\le K \le 0,$
%where \ $r=d(p,p_0)$ for some $p_0 \in M$.
%\vspace{2mm} -s(r) \le
If \ $\lim_{ \ r \rt \infty} \ r^2  s(r) =0$ when \ $dim \ M$ is
odd or $\int_0^\infty s(r)dr < \infty$ when \ $dim \ M$ is even, \
then $(M,g)$ is isometric to $\R^n$.} \vspace{3mm}

Results of both kinds have been extended to characterizing
hyperbolic manifolds. For instance, Min-Oo proved ~\cite{m} that a
spin $n$-manifold with scalar curvature $\ge -n(n-1)$ and
asymptotic to the hyperbolic metric in a strong sense must be
isometric to hyperbolic $n$-space. In the other direction, G.
Tian and Y. Shi recently proved ~\cite{ts}

%where $r$ again denotes distance to some point in $M$:

\vspace{3mm}

 \noindent {\bf Theorem 2 (Tian-Shi ~\cite{ts})}:
\noindent {\it Let $(M,g)$ be a simply-connected complete
Riemannian manifold of dimension $\ge 3$ \ with \ $K \le 0$ and
$Ricci \ge  -(n-1)$.
%$ -s(r)
%\le K \le 0,$
%where \ $r=d(p,p_0)$ for some $p_0 \in M$.
%\vspace{2mm}
If \ $\vert K+1 \vert = O(e^{-\alpha r})$ as $r \rt \infty$, for
some $\alpha
>2$, \ then $(M,g)$ is isometric to ${\mathbb H}^n$.}\vspace{3mm}

In this note, we prove two rigidity results by means of a simple
but versatile technique. The first result, Theorem A, is a direct
analogue of Theorem 1 above for characterizing hyperbolic space.
In fact, the same arguments will also give a quick proof of
Theorem 1 (For yet another proof of Theorem 1 involving the Tits
metric, see ~\cite{bgs}).
%The proof that we give is elementary and versatile: The same proof
%yields Theorem B below,

\vspace{3mm}

\noindent {\bf Theorem A:} \ {\it Let $(M,g)$ be a
simply-connected complete Riemannian manifold of dimension $\ge
3$. Suppose that \ $ -1-s(r) \ \le \ K \ \le \ -1$, where $r$
denotes distance to a fixed point in $M$. \vspace{1mm}

%(ii) $-1 \le K \le -1 +s(r)$ at all points of $M$,\vspace{2mm}
If \ $\lim_{ \ r \rt \infty} \ e^{2r}s(r) =0$, then $(M,g)$ is
isometric to ${\mathbb H}^n$.}

\vspace{3mm}

Theorem A complements Theorem 2 in the following sense: While
Theorem 2 implies rigidity for the {\it lower} bound $K \ge -1$,
%\ \le \ -1 + Ce^{\alpha r} \ \le \ 0$,
\ Theorem A gives rigidity for the {\it upper} bound $ K \le -1$.
Similarly, in the case of $\R^n$,  \ some results for $K \ge 0 $
were also proved in ~\cite{gw}.\vspace{2mm}

%In ~\cite{gw}, rigidity theorems for $\R^n$ are also proved for
%the lower curvature bound $ K \ge 0$. The corresponding results
%for the hyperbolic case, i.e., for $-1 \ \le \ K \ \le \ -1 +
%Ce^{\alpha r} \ \le \ 0$, where $C$ is any constant and $\alpha
%>2$, follow from Theorem 2. \vspace{2mm}
The second result, Theorem B, applies to geodesic balls in
Riemannian manifolds. It is valid in the presence of positive
sectional curvature. If $(M,g)$ is a Riemannian manifold and $S$
any subset of $M$, we write ``{\it $K=a$ \ on \ $S$}" to mean the
following: For any $q \in S$ and any 2-plane $P \subset T_qM$,
one has $K(P)=a$, where $K$ is the sectional curvature of $g$.

\vspace{3mm}

\noindent {\bf Theorem B:} {\it Let $(M,g)$ be a Riemannian
manifold of dimension $\ge 3$. Suppose that $K \ \le \ a $ \ on \
$B_p(R)$. \ If \ $a > 0$, assume that \ $R  \le \ min \{ \frac
{\pi}{2 \sqrt a}, \ inj(p) \}$.
%(ii) $-1 \le K \le -1 +s(r)$ at all points of $M$,\vspace{1.5mm}
\vspace{1mm}

If \ $K = a$ on \ $\partial B_p(R)$, \ then $K= a $ on \
$B_p(R)$.} \vspace{3mm}

Note that when $a > 0$, we do not need to assume that sectional
curvature has a fixed sign in the interior of the geodesic ball.
Note also that we are not only demanding that the sectional
curvatures achieve their maxima on $\partial B_p(R)$ but also that
all curvatures are equal (to $a$) on \ $\partial B_p(R)$.
Finally, we remark that the above theorem fails to hold if we
assume that $K \ge a$ \ instead of \ $K \le a$. An example is
given in Section 3.

\vspace{2mm}

%\begin{thm}\label{fla}
%{\bf Theorem (Greene-Wu):} {\it Let $(M,g)$ be a complete
%noncompact Riemannian manifold with $dim M \ge 3$ satisfying
%$$ -s(r) \le K \le 0,$$
%(ii) $0 \le K \le  s(r)$ and the exponential map is a covering map at $p$, \vspace{2mm}
%where $\lim_{ \ r \rt \infty} \ r^2  s(r) =0$. \vspace{2mm}

%Then $(M,g)$ is flat.}

%\end{thm}

The proof of these theorems is based on relative volume comparison
for distance spheres. The upper curvature bound implies that the
relative volume of distance spheres is increasing and $\ge 1$. On
the other hand, this bound also gives lower bounds for the
principal curvatures of the distances spheres. Combining this
with the lower bound on $K$ one sees that the intrinsic curvature
of the distance spheres is bounded below. Another application of
volume comparison shows that the relative volume approaches $1$
as $r \rt \infty$ (in Theorem A). Hence the relative volume is
equal to $1$ for all $r$ and one gets the required conclusion.

\section{proofs}

 We begin by recalling two standard results of comparison
geometry. Let $(M,g)$ be a simply-connected Riemannian manifold,
not necessarily complete. For $p \in M$ and $r \le inj(p)$, \
%let $\{r, \theta_i \}$ be the induced polar coordinates on $M$.
let $V(r)$ and $A(r)$ denote the volumes of the ball $B_p(r)$ and
the sphere $S_p(r)= \partial B_p(r)$ in $M$ and let $V^a(r)$, \
$S^a(r)$ denote the corresponding quantities in the
simply-connected space-form of curvature $a$, respectively. Let
$\rho(x):=d(p,x)$.
%Comparison theory for ODEs applied to the Riccatti
%equation for the shape operator $S$ of $S_p(r)$,
%$$\frac {dS}{dr}+S^2=R$$
%gives (see ~\cite{p}, for instance)
\vspace{3mm}

\noindent {\bf Hessian Comparison}: {\it  Let $(M^n,g)$, \ $n \ge
2$, \ be a Riemannian manifold and $p \in M$. Assume that $ K \le
a$ \ on \ $B_p(R)$, where $R \in (0, \infty]$ \ if \ $a \le 0$ \
and \ $R \ \le \ min \{ \frac {\pi}{2 \sqrt a}, \ inj(p) \}$ \ if
\ $a
>0$. \vspace{3mm}

If \ $\lambda$ \ is any eigenvalue of \ $Hess(\rho) $, \ then \
$\lambda \ \ge \lambda_a$, where
\[ \lambda_a(r) =
\left\lbrace
  \begin{array}{c l}
    {\sqrt {\vert a \vert} } \ {\coth (\sqrt {\vert a \vert} r)} & \text{if $a<0$},\\
    {r^{-1}} & \text{if $a=0$},\\
    {\sqrt a} \ {\cot (\sqrt a r)} & \text{if $a>0$}

    %\tfrac{\seq{A{f},h}}{\abs{\seq{A{f},h}}} & \text{otherwise}.
  \end{array}
\right. \]}

 \vspace{3mm}

\noindent {\bf Volume Comparison}: {\it Let $(M^n,g)$, \ $n \ge
2$, \ be a Riemannian manifold and $p \in M$. \vspace{2mm}

(i) \ Let \ $r \le R \le inj(p)$.  \ If \ $ K \le a$ \  on \
$B_p(R)$, \ then  \ $\frac {A(r)} {A_a(r)} $ \
%and \ $G(r):=\frac {A(r)} {A_a(r)}$ \ are
is an increasing function of \ $r$.
% \ $\lim_{r \rt 0} \ \frac {V(r)} {V_a(r)}=1$,
%F(r)= \lim_{r \rt 0} \ G(r)=1$, \ we have \ ${V(r)} \ge {V_a(r)}$ \ and \
%${A(r)} \ge {A_a(r)}$.
\ If \ $\lim_{r \rt R} \frac {A(r)} {A_a(r)}=1 $, \
%or \ $\lim_{r \rt R} G(r)=1$, \
then \ $K=a$ \ on \ $B_p(R)$. \vspace{2mm}

(ii) \ If \ $Ricci \ \ge \ (n-1)a$, \ then \ $V(r) \le V_a(r)$ for
any \ $r >0$.}

\vspace{6mm}

The proof begins with the  following linear algebra lemma:

\begin{lem}\label{sym}
Let $S: V \rt V$ be a positive semi-definite symmetric linear
operator on an inner-product space $V$. Let $$T(X,Y):=\frac { \li
S(X),X \ri \ \li S(Y),Y \ri -  \li S(X),Y \ri ^2}{\Vert X \Vert
^2 \Vert Y \Vert ^2 - \li X,Y \ri ^2}$$ for linearly independent
vectors $X$ and $Y$. Then
$$ \lambda ^2 \ \le \ T(X,Y) \ \le \ \mu^2,$$ where $\lambda$ and $\mu$ are the smallest and largest
eigenvalues of $S$, respectively.
\end{lem}

\begin{proof}
It can be checked that the $T(X,Y)$ depends only on the plane
spanned by $X$ and $Y$, i.e., $T(X,Y)=T(X',Y')$ if $X', Y'$ is
another basis for $ \ P=Span \{X,Y \}$.

  We claim that we can find orthogonal vectors \ $v,w \in P$ \ such
that $<S(v),w>=0$. This will clearly prove the lemma. In fact, let
$e_1, \ e_2$ be an orthonormal basis for $P$. We can assume that
$<S(e_1),e_2> \neq 0$, otherwise there is nothing to prove. Let
$v=e_1+ce_2, \ w=e_1 + de_2$, where $c,d$ are to be chosen. We
want
$$ <v,w>=1+cd=0, \ \ <S(v),w>= a_{22}cd + (c+d)a_{12}+a_{11}=0,$$
where $a_{ij}=<S(e_i),e_j>$. These equation give
$$ c^2 + (\frac {a_{11} -a_{22}}{a_{12}})c-1=0,$$
which can be solved to give the required $v$ and $w$.

\end{proof}
\vspace{2mm}
%\vspace{2mm}

The main ingredient in our proof is the following \vspace{2mm}

\noindent {\bf Key Lemma:} {\it Let $(M^n,g)$, \ $n \ge 3$, \ be a
Riemannian manifold. Suppose that \ $ a-s(r) \le K \le a$ \ on \
$B_p(R)$, \ where $R \in (0, \infty]$ \ if \ $a \le 0$ \ and \ $R
\ \le \ min \{ \frac {\pi}{2 \sqrt a}, \ inj(p) \}$ \ if \ $a
>0$. Then
%at $q=(r,\theta) \in \ S_r(p)$, we have $K \ge -1-s(r)$. Then
$$\frac {A(r)}{A_{a}(r)} \
%\sinh(r)^{n-1}}
\le \ (1-f_a(r)^2s(r))^{- \frac {n-1}{2}},$$ for all \ $r$ \ such
that \ $1-f_a(r)^2s(r) \ > \ 0$. \ Here
\[ f_a(r) =
\left\lbrace
  \begin{array}{c l}
    \frac {1}{\sqrt {\vert a \vert} } \sinh (\sqrt {\vert a \vert} r) & \text{if $a<0$},\\
    r & \text{if $a=0$},\\
    \tfrac {1}{\sqrt a} \sin (\sqrt a r) & \text{if $a>0$},\\
    %\tfrac{\seq{A{f},h}}{\abs{\seq{A{f},h}}} & \text{otherwise}.
  \end{array}
\right. \]} \vspace{2mm}

\begin{proof}
Let $\omega_n$ denote the volume of the unit sphere in $\R^n$. We
then have
\begin{equation}\label{on}
A^a(r)= f_a(r)^{n-1} \omega_n.
\end{equation}

We first consider the $a <0$ case.
%(i)Since $K \ge -1-s(r)$ at $q$, we have $K \ge  -1 -s(r) - \eps$
%in the ball $B_\delta(q)$ and $\eps \rt 0$ as $\delta \rt 0$.
By the Gauss-Codazzi equations for the curvature of the
submanifold $S_p(r)$ at a point $q \in S_p(r)$,
\begin{equation}\label{gc}
\tilde K (P) \ = \ K(P) \ + \ \li S(X),X \ri  \ \li S(Y),Y \ri -
\li S(X),Y \ri ^2,
\end{equation}
for any 2-plane $P \subset T_q S_p(r)$. Here $\{X, Y\}$ is any
orthonormal basis of $P$. \vspace{2mm}

Now, since $<S(X),Y> \ = \ Hess(\rho)(X,Y)$, we can apply the
Hessian comparison theorem to estimate the eigenvalues of $S$.
Combining these estimates with Lemma ~\ref{sym} and using the
condition $K \ge a-s(r)$ in (~\ref{gc}) we get
$$\tilde K \ \ge \  a-s(r)  \ + \  {\vert a \vert} \ \coth (\sqrt {\vert a \vert}r)^2.$$
This implies that
\begin{equation}\label{sec}
%{\vert a \vert} ^{-1} \sinh(\sqrt {\vert a \vert}r)^2 \tilde K \
f_a(r)^2 \tilde K \ \ge \ k(r):= 1 \ - {\vert a
\vert}^{-1}\sinh(\sqrt {\vert a \vert}r)^2 s(r).
\end{equation}

Let us fix an $r$ with $k(r) >0$.
%\ $r$ \ with \ $ a^{-1} \sinh^2(r) s(\sqrt {\vert a
%\vert}r) < 1$.
\ By (~\ref{sec}), we can apply the Bonnet-Myers theorem to the
Riemannian manifold $(N,h)= (S_p(r), \ f_a(r)^{-2}g)$ to get
$$ diam \ N \ \le \  {\pi} k(r)^{- \frac {1}{2}}.$$

From (ii) of the volume comparison theorem, we have: Volume of $N
=$ Volume of \ $B^N_p({\pi}k(r)^{- \frac {1}{2}})  \ \ \le \
k(r)^{- \frac {n-1}{2}} \ \omega_n$. \ Here $B^N$ denotes balls in
$(N,h)$.

Since Volume of $N=f_a(r)^{-(n-1)} \ A(r)$, we have \ $A(r) \ \le
\ f_a(r)^{n-1} k(r)^{- \frac {n-1}{2}} \ \omega_n $. \ Combining
this with (~\ref{on}) gives the required inequality. \vspace{3mm}

The proof goes through without any changes for \ $a=0$. When $a
>0$, \  note that $\lambda \ge \lambda_a \ge 0$ \ as long as \ $r \le
min \{ \frac {\pi}{2 a}, \ inj(p) \}$. Hence Lemma ~\ref{sym} can
be applied and the rest of the proof goes through.
\end{proof} \vspace{3mm}

\noindent {\it Proof (of Theorems A and B)}: We start with the
proof of Theorem A.

Since $K \le -1$, by the volume comparison theorem, the ratio $
F(r)=\frac {A(r)}{A_{-1}(r)} \ge 1$ is a non-decreasing function
of $r$. By the Key Lemma and the hypothesis that $e^{2r} s(r) \rt
0 $ as $r \rt \infty$, \ we see that $\lim_{r \rt \infty} F(r) \le
1$. Hence $F(r)=1$ for all $r
>0$. By the equality part of the volume comparison theorem we
obtain $K=-1$ \ on \ $M$.

The proof of Theorem B is similar. In this case the function
$F(r)=\frac {A(r)}{A_{a}(r)} \ge 1$ is increasing for $r \le R$.
\ Since $K=a$ on $\partial B_p(r)$,  we have \ $K \ge a-s(r)$ \
where \ $s(r) \rt 0$ \ as \ $r \rt R$. \ Combining this with the
Key Lemma and arguing as before, we get $K=a$ \ for \ $r \le R$.

\hfill $ \square $

\section{remarks}
The remarks below concern the validity of the theorems under lower
bounds on $K$. \vspace{2mm}

(i) As mentioned earlier, an analogue of Theorem A for the bounds
\ $-1 \ \le \ K \ \le \ -1+Ce^{-\alpha r} \le 0$ \ with \ $C>0, \
\alpha >2$ is implied by the result in ~\cite{ts}. \vspace{2mm}

%Our proof can recover this result when $dim M$ is odd, as
%follows: Assume that $-1 \le K -1-s(r) \le 0$, where $\lim_{r \rt
%\infty}e^{2r}s(r)=0$. The Hessian comparison theorem in this case
%would give {\it upper} bounds for the principal curvatures while
%the volume comparison theorem would state that the relative
%volume is {\it decreasing} (and $\le 1$).
%
%However, it is not clear that the key lemma holds: We do get
%$\tilde K \le a - s(r)+\lambda_a(r)^2$ and $\tilde K> 0$ on
%$S_p(R)$.
%
%In the absence of an estimate on the injectivity radius of
%$S_p(r)$, we cannot conclude anything about its volume. Note that
%we do have such an estimate if dim $S_p(r)$ is even, i.e., dim
%$M$ is odd, by a result of Klingenberg ~\cite{k}, since $\tilde K
%>0$. Hence, {\it if dim M is odd}, \ our proof yields: {\it
%Theorem A is valid if \ $-1 \le K \le -1-s(r) \le 0$, \ where \
%$\lim_{r \rt \infty} e^{2r}s(r)=0$} \ and \ {\it Theorem B is
%valid if \ $ K \ge a$ .} \vspace{2mm}

(ii) Theorem B is no longer true under the bound \ $K \ge a$.
Indeed, consider the metric \ $g=dr^2 + f(r)^2 g_0$ \ on the ball
$D= \{ x \ : \ r =\Vert x \Vert < \frac {\pi}{2} \}$ in $\R^n$,
where $g_0$ is the standard round metric on $S^{n-1}$ and
\[ f(r) = \left\lbrace
  \begin{array}{c l}
    {\sin (r)} & \text{if \ $r < c - \eps$},\\
    {h(r)} & \text{if \ $c - \eps \ \le \ r  \ \le \ c + \eps$},\\
    {-r+ \frac {\pi}{2}} & \text{if \ $c+ \eps < \ r \ < \frac {\pi}{2}$.}

    %\tfrac{\seq{A{f},h}}{\abs{\seq{A{f},h}}} & \text{otherwise}.
  \end{array}
\right. \] Here $h$, $\eps$ and $c$ are to be chosen. Let $c \in
(0,\frac {\pi}{2})$ be the solution to \ $\sin(r)\ = \ -r +\frac
{\pi}{2}$. \ Choose $\eps$ so that $[c- \eps,c+ \eps] \subset
(0,\frac {\pi}{2})$.

Let $h:[c-\eps,c+ \eps] \rt (0,\infty)$ be a smooth function with
$$h'' \le 0 \ \ {\rm and} \ \ -1 \le h'  \le 1$$ which agrees (up to second order) with
$\sin(r)$ at $c-\eps$ and with $-r+\frac {\pi}{2}$ at $c+\eps$.
Since the sectional curvatures of the metric $g=dr^2+f(r)^2g_0$
lie between the values of
$$ - \frac {f''(r)}{f(r)} \ \ {\rm and} \ \ \frac
{1-f'(r)^2}{f(r)^2},$$ we see that $g$ has $K \ge 0$ everywhere
and $K=0$ on $\partial B_0(c+\eps)$. On the other hand, $K=1$ on
$B_0(c- \eps)$. \ Hence Theorem  B fails to hold.


\vspace{2mm}


\begin{thebibliography}{10}

\bibitem{bgs} W. Ballmann, M. Gromov, V. Schroeder
\textit{Manifolds of nonpositive curvature}, Progress in
Mathematics, \textbf{61}. Birkh\"auser Boston, Inc., Boston, MA,
1985.


%\bibitem{ce}J. Cheeger, D. Ebin
%\textit{Comparison theorems in Riemannian geometry},
%North-Holland Mathematical Library, Vol. 9. North-Holland
%Publishing Co., New York, 1975.


\bibitem{gw} R. E. Greene, H. Wu
\textit{Gap theorems for noncompact Riemannian manifolds}, Duke
Math. J. \textbf{49} (1982), no. 3, 731-756.

%\bibitem{k} W. Klingenberg
%\textit{Contributions to Riemannian geometry in the large}, Ann.
%of Math. (2) \textbf{69} 1959, 654-666.

\bibitem{m} M. Min-oo
\textit{Scalar curvature rigidity of asymptotically hyperbolic
spin manifolds}, Math. Ann. \textbf{285} (1989), no. 4, 527-539.

%\bibitem{p} P. Petersen
%\textit{Riemannian geometry}, Graduate Texts in Mathematics,
%\textbf{171}, Springer-Verlag, New York, 1998.

\bibitem{scy} R. Schoen, S-T. Yau
\textit{On the proof of the positive mass conjecture in general
relativity}, Comm. Math. Phys. \textbf{65} (1979), no. 1, 45-76.

\bibitem{sy} Y. T. Siu, S-T. Yau,
\textit{Complete K\"ahler manifolds with nonpositive curvature of
faster than quadratic decay}, Ann. of Math. (2) \textbf{105}
(1977), no. 2, 225--264.

\bibitem{ts} G. Tian, Y. Shi
\textit{Rigidity of asymptotically hyperbolic manifolds}, Comm.
Math. Phys. \textbf{259} (2005), no. 3, 545-559.

\end{thebibliography}
\end{document}